\documentclass[12pt,a4paper]{amsart}

\usepackage[utf8]{inputenc}
\usepackage[T1]{fontenc}
\usepackage[english]{babel}
\usepackage{cite}

\usepackage{indentfirst}
\usepackage{amssymb}
\usepackage{amsfonts}
\usepackage{amsmath}
\usepackage{amsthm}
\usepackage[top=2.5cm, bottom=2.5cm, left=2.5cm, right=2.5cm]{geometry}
\usepackage{amsopn}
\usepackage[colorlinks]{hyperref}
\usepackage{mathrsfs}
\usepackage{amsxtra}
\usepackage{color}
\usepackage{dsfont}

\newcommand{\al}{\alpha}
\newcommand{\be}{\beta}

\theoremstyle{plain}
\newtheorem{thm}{Theorem}

\newtheorem{lem}[thm]{Lemma}
\newtheorem{prop}[thm]{Proposition}
\newtheorem{cor}[thm]{Corollary}

\theoremstyle{definition}

\newtheorem*{example*}{Example}

\newtheorem*{rem*}{Remark}
\newtheorem{rem}[thm]{Remark}

\newcommand{\sgn}{\text{sgn}}
\newcommand{\dd}{d_{\Omega}}
\newcommand{\meas}{\left|\mathbb{S}^{d-1}\right|}

\newcommand{\R}{\mathbb{R}}

\DeclareMathOperator{\diam}{diam}

\DeclareMathOperator{\ucodim}{\underline{co\,dim}}

\DeclareMathOperator{\supp}{supp}
\DeclareMathOperator{\dist}{dist}

\DeclareSymbolFont{bbsymbol}{U}{bbold}{m}{n}
\DeclareMathSymbol{\ind}{\mathbin}{bbsymbol}{'061}

\title[Asymptotics of weighted Gagliardo seminorms]{Asymptotics of weighted Gagliardo seminorms}
\author[M{.} Kijaczko]{Micha\l{} Kijaczko}


\address[ M.K.]{Faculty of Pure and Applied Mathematics\\ Wroc{\l}aw University 
	of Science and Technology\\
	Wybrze\.ze Wyspia\'nskiego 27,
	50-370 Wroc{\l}aw, Poland
}
\email{michal.kijaczko@pwr.edu.pl}

\begin{document}

\begin{abstract}
In this paper, we consider fractional Sobolev spaces equipped with weights being powers of the distance to the boundary of the domain. We prove the versions of Bourgain--Brezis--Mironescu and Maz'ya--Shaposhnikova asymptotic formulae for weighted fractional Gagliardo seminorms. For $p>1$ we also provide a nonlocal characterization of classical weighted Sobolev spaces with power weights.

\end{abstract}

	\maketitle
\section{Introduction and preliminaries}
For $d\geq 1$, let $\Omega\subset\R^d$ be a nonempty open set, $0<s<1$ and $p\geq$1. Recall that the fractional Gagliardo seminorm of a measurable function $f\colon\Omega\to\R$ is given by
\begin{equation}\label{gagliardoseminorm}
[f]_{W^{s,p}(\Omega)}:=\left(\int_{\Omega}\int_{\Omega}\frac{|f(x)-f(y)|^p}{|x-y|^{d+sp}}\,dy\,dx\right)^{\frac{1}{p}}.
\end{equation}
The fractional Sobolev space $W^{s,p}(\Omega)$ is then defined as
$$
W^{s,p}(\Omega)=\{f\in L^p(\Omega):[f]_{W^{s,p}(\Omega)}<\infty\}.
$$
The celebrated Bourgain--Brezis--Mironescu theorem \cite{BBM01} gives the asymptotic of the integral form \eqref{gagliardoseminorm}, as $s\rightarrow 1^-$:  if $\Omega$ is a $W^{1,p}$ - extension domain, then, for $p>1$,
\begin{equation}\label{bbm}
\lim_{s\rightarrow 1^-}(1-s)\int_{\Omega}\int_{\Omega}\frac{|f(x)-f(y)|^p}{|x-y|^{d+sp}}\,dy\,dx=K_{d,p}\int_{\Omega}\left|\nabla f(x)\right|^p\,dx,
\end{equation}
whenever $f\in L^p(\Omega)$. The constant $K_{d,p}$ appearing in \eqref{bbm} is equal to
\begin{equation}\label{Kdp}
 K_{d,p}=\frac{2\pi^{\frac{d-1}{2}}}{p}\frac{\Gamma\left(\frac{p+1}{2}\right)}{\Gamma\left(\frac{p+d}{2}\right)}.   
\end{equation}
For $p=1$, \eqref{bbm} also holds, provided that $f\in W^{1,1}(\Omega)$. D\'{a}vila \cite{MR1942130} showed, that in this case, for sufficiently regular $\Omega$ and $f\in L^1(\Omega)$, an analogous convergence is true for the Sobolev seminorm replaced by the $\text{BV}$-seminorm of $f$ on the right-hand side of \eqref{bbm}. Recent results extend the Bourgain--Brezis--Mironsescu formula to the setting of  fractional Orlicz--Sobolev spaces \cite{MR4215683}, fractional Sobolev spaces with variable exponents \cite{Minhyun}, Triebel--Lizorkin spaces \cite{MR4686398} or more general Lévy-type measures \cite{MR4565919}.

Another famous result of this kind was established by Maz'ya and Shaposhnikova \cite{MR1940355}. They proved that
\begin{equation}\label{mazyashapo}
\lim_{s\rightarrow 0^+}s\int_{\R^d}\int_{\R^d}\frac{|f(x)-f(y)|^p}{|x-y|^{d+sp}}\,dy\,dx=\frac{2}{p}\meas\int_{\R^d}|f(x)|^p\,dx.
\end{equation}
Here $\meas=\frac{2\pi^{d/2}}{\Gamma(d/2)}$ is the surface measure of the unit sphere in $\R^d$ and it is assumed that $f\in\bigcup_{0<s<1}W^{s,p}(\R^d)$. This limit has been also investigated  in the context of fractional Orlicz--Sobolev spaces \cite{MR4165063}  or in the magnetic setting \cite{MR3601583}. 

In this paper, we are interested in weighted fractional Gagliardo seminorms. The weights that are taken into consideration are of the particular form: they are powers of the distance to the boundary of the domain. Let us now make this precise. For $\Omega$ being an open, nonempty, proper subset of $\R^d$, $f\colon\Omega\to\R$ a measurable function, $p\geq 1$, $0\leq s<1$ and $\al,\be\in\R$ we define the weighted Gagliardo seminorm as
\begin{equation}\label{weightedgagliardoalfabeta}
[f]_{W^{s,p}_{\al,\be}(\Omega)}:=\left(\int_{\Omega}\int_{\Omega}\frac{|f(x)-f(y)|^p}{|x-y|^{d+sp}}d_{\Omega}(x)^{-\al}d_{\Omega}(y)^{-\be}\,dy\,dx\right)^\frac{1}{p}.    
\end{equation}
Here $d_{\Omega}(x):=\dist(x,\partial\Omega)=\inf_{y\in\partial\Omega}|x-y|,\,x\in\R^d$, denotes the distance to the boundary. The weighted fractional Sobolev space is then defined as
$$
W^{s,p}_{\al,\be}(\Omega)=\{f\in L^p(\Omega, \dd^{-\al-\be}):[f]_{W^{s,p}_{\al,\be}(\Omega)}<\infty\}
$$
and we endow it with the natural norm $$\|f\|_{W^{s,p}_{\al,\be}(\Omega)}=\|f\|_{L^p\left(\Omega, \dd^{-\al-\be}\right)}+[f]_{W^{s,p}_{\al,\be}(\Omega)}.$$

It is worth noticing here that, for unbounded domains, unlike in the unweighted case we may allow to have $s=0$ in \eqref{weightedgagliardoalfabeta} and keep the expression finite. Weighted Gagliardo seminorms with power-type weights have been widely studied in past years. For example, we refer to \cite{MR3420496} and \cite{MK} for the investigation of the problem of density of smooth, compactly supported functions in the weighted fractional Sobolev space related to \eqref{weightedgagliardoalfabeta}, or to \cite{ID} for the results on interpolation spaces. Fractional Sobolev spaces with the weight $\min\{\dd(x),\dd(y)\}^{2\al}$ appear in \cite{MR3620141} while studying nonlocal equations involving fractional Laplacian. Recent results are devoted to weighted fractional Hardy inequalities \cite{sharpweighted}, \cite{sharpweighted1<p<2}.

Recall that, for $\gamma\in\R$, the classical Sobolev space with power-type weight $d_{\Omega}^{-\gamma}$ is defined as
$$
W^{1,p}_{\gamma}(\Omega):=\left\{f\in L^p(\Omega, \dd^{-\gamma}):\int_{\Omega}\frac{|\nabla f(x)|^p}{\dd(x)^{\gamma}}\,dx<\infty\right\},
$$
where the gradient is understood in the distributional sense. Of course, $W^{1,p}_{\gamma}(\Omega)$ is equipped with the natural Sobolev-type norm and analogously we define the weighted Sobolev space $W^{1,p}(\Omega, w)$ for any weight $w$. These spaces are well-known; we refer the  reader  interested with details to the book by Kufner \cite{MR664599}. Weighted Sobolev spaces play a crucial role in the theory of partial differential equations.  We will call $\Omega$ a $W^{1,p}_{\gamma}$ - extension domain if there exists a bounded linear operator $$\Lambda\colon W^{1,p}_{\gamma}(\Omega)\to W^{1,p}(\R^d,\dd^{-\gamma})$$ (notice that $\dd$ is defined also outside of the domain $\Omega$) such that $\Lambda f \big|_{\Omega}=f$ and $$\|\Lambda f\|_{W^{1,p}(\R^d,\dd^{-\gamma})}\leq C\|f\|_{W^{1,p}_{\gamma}(\Omega)}$$ for all $f\in W^{1,p}_{\gamma}(\Omega)$, with $C$ being a universal constant. We refer to Section \ref{Section3} for a more detailed discussion on the extension problem.

In this note, we study the limit behavior of the weighted fractional Gagliardo seminorms. We establish limits involving $W^{0,p}_{\al,\al}(\R^d)$ - seminorms, when $\al\rightarrow 0^+$, or $\al\rightarrow d^-$, being a~ weighted analogue of Maz'ya--Shaposhnikova formula. We also provide the asymptotic of $W^{s,p}_{\al,\be}$ - seminorm, as $s\rightarrow 1^-$, in the spirit of Bourgain, Brezis and Mironescu theorem. Using this result, for $p>1$ we prove a nonlocal characterization of classical Sobolev spaces with power weights on extension domains.
\section{Main results}

\subsection{Bourgain--Brezis--Mironescu-type asymptotics for weighted Gagliardo seminorms}

\begin{thm}\label{bbmweighted}
 Let $\Omega\subset\R^d$ be nonempty and open. Let $f\in C_c^2(\overline{\Omega})$, $p,d\geq 1$ and let $U\subset\R^2$ be a nonempty open set. 
Assume that, for given $f$, the seminorms \eqref{weightedgagliardoalfabeta} are finite for all $(\al,\be)\in U$ and for all $s\in (s_0,1)$ for some $s_0\in(0,1)$. Then, 
 \begin{equation}\label{limits1}
 \lim_{s\rightarrow 1^-}(1-s)\int_{\Omega}\int_{\Omega}\frac{|f(x)-f(y)|^p}{|x-y|^{d+sp}}d_{\Omega}(x)^{-\al}d_{\Omega}(y)^{-\be}\,dy\,dx=K_{d,p}\int_{\Omega}\frac{|\nabla f(x)|^p}{d_{\Omega}(x)^{\al+\be}}\,dx.    
 \end{equation}
 Moreover, if $\Omega$ is a $W^{1,p}_{\al+\be}$ - extension domain, $p>1$, $\dd^{-\al-\be}$ belongs to the Muckenhoupt class $A_p$ and \eqref{condition1} or \eqref{condition2} is satisfied, then \eqref{limits1} holds for all $f\in L^{p}(\Omega, \dd^{-\al-\be})$, with the convention that the right-hand side is infinite, if $f\notin W^{1,p}_{\al+\be}(\Omega)$. In this case we provide a~ nonlocal characterization of the Sobolev space $W^{1,p}_{\al+\be}(\Omega)$: if and only if
 $$
\liminf_{s\rightarrow 1^{-}}(1-s)\int_{\Omega}\int_{\Omega}\frac{|f(x)-f(y)|^p}{|x-y|^{d+sp}}d_{\Omega}(x)^{-\al}d_{\Omega}(y)^{-\be}\,dy\,dx<\infty,
$$
then $f\in W^{1,p}_{\al+\be}(\Omega)$. For $p=1$, \eqref{limits1} holds if $f\in W^{1,1}_{\al+\be}(\Omega)$.
\end{thm}
A natural question arises at first, namely, what are the conditions for the parameters and the domain to have the weighted Gagliardo seminorms finite? For the discussion on the range of parameters in the case when $\Omega=\R^d\setminus\{0\}$, see Section \ref{Section3}. Moreover, it is shown in \cite[Section 3]{MK}, that for $C_c^1$ functions and $\Omega$ bounded, \eqref{weightedgagliardoalfabeta} is finite when $0<s<1$ and $\al+\be<d-\overline{\text{dim}}_M(\partial\Omega)$, where $\overline{\text{dim}}_M(\partial\Omega)$ is the upper Minkowski dimension of the boundary, see \cite[Section 2]{MR4144553}. Additionaly, if $\Omega$ is bounded and uniform \cite{10.2748/tmj/1178228081}, then, according to \cite[Lemma 16]{MK},  \eqref{weightedgagliardoalfabeta} is finite when $\al,\be<\ucodim_A(\partial\Omega)$ and $\al+\be<d-\overline{\text{dim}}_M(\partial\Omega)+p(1-s)$. Here $\ucodim_A(\partial \Omega)$ stands for the lower Assouad codimension of the boundary; the definition of the latter may be found in \cite[Section 3]{MR3205534} or Section \ref{Assouad}.  All the mentioned examples fit into the assumptions of Theorem \ref{bbmweighted}. For Lipschitz domain we have $\ucodim_A(\partial\Omega)=d-\overline{\text{dim}}_M(\partial\Omega)=1$.

It is worth to point out that in the first part of the Theorem \ref{bbmweighted} we do not assume any regularity of the domain $\Omega$. The reason for it is that the assumption that $\al,\be\in U$, where $U$ is open, gives us a certain flexibility to work with exponents of weights. More precisely, the proof of Bourgain--Brezis--Mironescu formula is based on the fact that if $\Omega$ is an extension domain and $f$ is smooth, then, for $0<\theta<1$,
\begin{equation}\label{limitregular}
 \lim_{s\rightarrow 1^-}(1-s)\int_{\Omega}\int_{|x-y|>\theta \dd(x)}\frac{|f(x)-f(y)|^p}{|x-y|^{d+sp}}\,dy\,dx=0.   
\end{equation}
However, \cite[Example 2.1]{MR3865123} reveals that for irregular domains, \eqref{limitregular} does not hold in general. The presence of weights and compactness of the support of the function $f$ removes this difficulty, see the proof of Theorem \ref{bbmweighted}.

In particular, Theorem \ref{bbmweighted} applied to the case when $\Omega=\R^d\setminus\{0\}$ gives us the following result.
\begin{cor}\label{corRd}
Let $f\in C_c^2(\R^d)$ and $-p<\al,\be<d$, $\al+\be<d$. Then,
\begin{equation}\label{corollaryRd}
 \lim_{s\rightarrow 1^-}(1-s)\int_{\R^d}\int_{\R^d}\frac{|f(x)-f(y)|^p}{|x-y|^{d+sp}|x|^{\al}|y|^{\be}}\,dy\,dx=K_{d,p}\int_{\R^d}\frac{\left|\nabla f(x)\right|^p}{|x|^{\al+\be}}\,dx.
 \end{equation}
 Moreover, if $f\in L^p(\R^d,|\cdot|^{-\al-\be})$, when $p>1$, or $f\in W^{1,1}_{\al+\be}(\R^d)$, when $p=1$, then the above holds if, in addition, $\,-(p-1)d<\al+\be<d$ (or $0\leq\al+\be<d$, if $p=1$) and either $0\leq\al,\be<d$ or $\al\be\geq 0$ and $0\leq\al+\be<d$.
\end{cor}

\subsection{Asymptotics of $W^{0,p}_{\al,\al}(\R^d)$ - seminorms for $\al\rightarrow 0^+$ and $\al\rightarrow d^-$}

Below are our two results concerning the asymptotic behavior of the $W^{0,p}_{\al,\al}(\R^d)$ - seminorm.
\begin{thm}\label{thm2}
Let $p,d\geq 1$ and $f\in C_c^{1}(\R^d)$. Then
$$
\lim_{\alpha\rightarrow 0^+}\al\int_{\R^d}\int_{\R^d}\frac{|f(x)-f(y)|^p}{|x-y|^d|x|^{\al}|y|^{\al}}\,dy\,dx=2\meas\int_{\R^d}|f(x)|^p\,dx.
$$
\end{thm}

\begin{thm}\label{mazyashapolimitd}
Let $p,d\geq 1$, let $f$ be of class $C^1$, $\displaystyle\lim_{x\rightarrow\infty}f(x)$ exists and is finite. Moreover, suppose that the support of $f$ does not contain the origin (in particular, we may have $f\in C_c^1(\R^d\setminus\{0\})$). Then,
$$
\lim_{\al\rightarrow d^-}\left(d-\al\right)\int_{\R^d}\int_{\R^d}\frac{|f(x)-f(y)|^p}{|x-y|^d|x|^{\al}|y|^{\al}}\,dy\,dx=2\meas\int_{\R^d}\frac{|f(x)|^p}{|x|^{2d}}\,dx.
$$
\end{thm}
\section{Proofs}\label{Section3}
\subsection{Weighted fractional seminorms for the whole space}\label{Section 3.1}
Let us start with considerations on the finiteness of the weighted Gagliardo seminorms corresponding to the whole space $\R^d$. 

For $\Omega=\R^d\setminus\{0\}$ the seminorm \eqref{weightedgagliardoalfabeta} takes the form
\begin{equation}\label{seminormRd}
 [f]^p_{W^{s,p}_{\al,\be}(\R^d)}=\int_{\R^d}\int_{\R^d}\frac{|f(x)-f(y)|^p}{|x-y|^{d+sp}|x|^{\al}|y|^{\be}}\,dy\,dx.
\end{equation}
Dipierro and Valdinoci \cite[Lemma 2.1]{MR3420496} proved, that for $f\in C_c^1(\R^d)$, the expression \eqref{seminormRd} is finite, when $d,p\geq 1$, $0<s<1$, $-sp<\al,\be<d$ and $\al+\be<d$. Their proof works also in the case $s=0$. However, when the support of $f$ does not contain the origin, the range of parameters may be extended. To see this, let us first recall that the inversion $T\colon\R^d\setminus\{0\}\to\R^d\setminus\{0\}$ is defined by
$$
T(x)=\frac{x}{|x|^2}.
$$
It is well known that $T$ is an involution, $|T(x)|=\frac{1}{|x|}$, $|T(x)-T(y)|=\frac{|x-y|}{|x| |y|}$ and the Jacobian of $T$ is equal to $(-1)^d|x|^{-2d}$. Moreover, slightly abusing the notation, we can associate with the mapping $T$ a linear operator $T\colon C_c\left(\R^d\setminus\{0\}\right)\to C_c\left(\R^d\setminus\{0\}\right)$ by the formula
$$
Tf:=f\circ T.
$$
Indeed, denoting $B_\rho=\{x\in\R^d:|x|<\rho\}$, if $\supp f\subset B_R\setminus B_r$, then $\supp Tf\subset B_{\frac{1}{r}}\setminus B_{\frac{1}{R}}$. Of course, $Tf$ makes sense for all functions defined on $\R^d\setminus\{0\}$ or even $\R^d$, provided that $f$ has a limit in infinity. Hence, when $Tf\in C_c^1(\R^d)$, by a substitution $x\mapsto T(x),\,y\mapsto T(y)$ we get that
$$
[f]_{W^{s,p}_{\al,\be}(\R^d)}=[Tf]_{W^{s,p}_{\al',\be'}(\R^d)},
$$
where $\al'=d-\al-sp$, $\be'=d-\be-sp$, and the latter seminorm is finite when $-sp<\al,\be<d$, but $d-2sp<\al+\be<2d-sp$. In particular, for $s=0$, we obtain that \eqref{seminormRd} is finite for all $0<\al,\be<d$ such that $\al+\be\neq d$. The finiteness in the case when $\al+\be=d$ may be then showed by a simple limit argument.  The above explains why we should seek the asymptotics of $W^{0,p}_{\al,\al}(\R^d)$ - seminorms, when $\al\rightarrow 0^+$ and $\al\rightarrow d^-$.
\subsection{The lower Assouad codimension}\label{Assouad}
For given nonempty $E\subset\R^d$ the definition of the lower Assouad codimension may be found for example in \cite[Section 3]{MR3205534}. Below we provide an equivalent definition, namely the Aikawa condition, see \cite[Theorem 5.1]{MR3055588} and \cite[Remark 3.2]{MR3900847}. According to this, the lower Assouad codimension of $E$ is defined as the supremum of all exponents $\rho\geq 0$, for which there exists a constant $C\geq 1$ such that
\begin{equation}\label{Aikawa}
\int_{B(x,r)}\dist(y,E)^{-\rho}\,dy\leq Cr^{-\rho}\left|B(x,r)\right|,
\end{equation}
for all $x\in E$ and $r>0$. 

We also have $\ucodim_A(E)=d-\overline{\text{dim}}_A(E)$, where the latter is the more recognizable upper Assouad dimension. To give some concrete examples, $\ucodim_A(\{0\})=d$ and for a Lipschitz domain, the lower Assouad codimension of its boundary is equal to one. Moreover, from the Aikawa condition \eqref{Aikawa} it follows that the function $\dist(x,E)^{-\rho}$ is locally integrable near $E$, if $\rho<\ucodim_A(E)$. In particular, $\dd^{-\rho}\in L^1_{loc}(\overline{\Omega})$ if $\rho<\ucodim_A(\partial\Omega)$.

\subsection{Extensions and embeddings} In order to prove the second part of the Theorem \ref{bbmweighted}, we need to ensure that we are in the case, when $W^{1,p}_{\al+\be}(\Omega)\subset W^{s,p}_{\al,\be}(\Omega)$. This is of course not true in general; for example, we always have $C_c^1(\Omega)\subset W^{1,p}_{\al+\be}(\Omega)$, but not necessarily $C_c^1(\Omega)\subset W^{s,p}_{\al,\be}(\Omega)$, see \cite[Remark 17]{MK}. However, the desired inclusion will hold, provided, among others, that $\Omega$ is a $W^{1,p}_{\al+\be}$ - extension domain. For the general nonnegative and locally integrable weight $w$ on $\R^d$, the problem of extending the function from the weighted Sobolev space $W^{1,p}(\Omega,w)$ to $W^{1,p}(\R^d,w)$ was investigated by Chua \cite{MR1206339,MR1245837,MR2223454}. In virtue of \cite[Theorem 1.1]{MR1206339}, for $p\geq 1$ an extension operator for $W^{1,p}_{\gamma}(\Omega)$ exists, if the domain is sufficiently regular (for example Lipschitz) and the corresponding weight $\dd^{-\gamma}$ belongs to the Muckenhoupt class $A_p$. Recall that $w\in A_p$ for $p>1$, if
$$
\sup\frac{1}{|Q|}\left(\int_{Q}w(x)\,dx\right)^\frac{1}{p}\left(\int_{Q}w(x)^{-\frac{1}{p-1}}\,dx\right)^{\frac{p-1}{p}}<\infty,
$$
where the supremum is taken over all cubes $Q\subset\R^d$. Moreover, $w\in A_1$ if there exists a universal constant $C$ such that
$$
\frac{1}{|Q|}\int_{Q}w(x)\,dx\leq C w(y),\,\text{for almost all }y\in Q,
$$
for all cubes. It is straightforward to check that one always has the inclusion $A_1\subset A_p$.

The work \cite{MR3900847} provides a complete characterization of the condition $\dd^{-\gamma}\in A_p$ to be fulfilled. More precisely, by the result of\cite[Theorem 1.1]{MR3900847}, $\dd^{-\gamma}\in A_p$ for $p>1$ if and only if $-(p-1)\ucodim_A(\partial\Omega)<\gamma<\ucodim_A(\partial\Omega)$ and $\dd^{-\gamma}\in A_1$ if and only if $0\leq\gamma<\ucodim_A(\partial\Omega)$. In particular, since $\ucodim_A(\{0\})=d$, one recovers the classical fact that $|x|^{-\gamma}\in A_p$ if and only if $-(p-1)d<\gamma<d$ for $p>1$ and $|x|^{-\gamma}\in A_1$ if and only if $0\leq\gamma<d$.
\begin{prop} Let $\Omega$ be a $W^{1,p}_{\al+\be}$ - extension domain, $\dd^{-\al-\be}\in A_p$ and $p>1$. If in addition either
\begin{equation}\label{condition1}
\dd^{-\al}\in A_1,\,\dd^{-\be}\in A_1
\end{equation}
or
\begin{equation}\label{condition2}
\al\be\geq 0,\,\dd^{-\al-\be}\in A_1,
\end{equation}
 then there exists a constant $C=C(d,s,p,\al,\be,\Omega)$ such that
 $$
 \|f\|_{W^{s,p}_{\al,\be}(\Omega)}\leq C\|f\|_{W^{1,p}_{\al+\be}(\Omega)}
 $$
 for $f\in W^{1,p}_{\al+\be}(\Omega)$. In consequence, we have the embedding $W^{1,p}_{\al+\be}(\Omega)\subset W^{s,p}_{\al,\be}(\Omega)$ for all $0<s<1$. Moreover, one has 
 $$\limsup_{s\rightarrow 1^-}(1-s)[f]^p_{W^{s,p}_{\al,\be}(\Omega)}<\infty.$$
\end{prop}
\begin{proof}
 Let $f\in W^{1,p}_{\al+\be}(\Omega)$ and let $\Lambda f$ be its extension. By the version of Meyers--Serrin theorem for weighted Sobolev spaces with Muckenhoupt weights \cite[Theorem 2.5]{MR1246890}, we may assume that $f\in C^{\infty}(\Omega)$. Moreover, by \cite[Theorem 1.1]{MR1245837}, $\Lambda f\in C^{\infty}(\R^d)$. We first assume \eqref{condition1} to be satisfied. Then, we can estimate the weighted Gagliardo seminorm as follows, 
\begin{align*}
 [f]^p_{W^{s,p}_{\al,\be}(\Omega)}&=\left(\int_{\Omega}\int_{|x-y|>1}+\int_{\Omega}\int_{|x-y|<1}\right)\frac{|f(x)-f(y)|^p}{|x-y|^{d+sp}}\dd(x)^{-\al}\dd(y)^{-\be}\,dy\,dx\\
 &=:I_1+I_2.
\end{align*}
For  $u\in L^1_{loc}(\R^d)$, let 
$$Mu(x)=\sup_{r>0}r^{-d}\int_{B(x,r)}u(y)\,dy$$ 
be the Hardy--Littlewood maximal operator. For $I_1$ we have the following bound,
\begin{align*}
I_1\leq & 2^{p-1}\int_{\Omega}\int_{|x-y|>1}\frac{|f(x)|^p}{|x-y|^{d+sp}}\dd(x)^{-\al}\dd(y)^{-\beta}\,dx\,dy\\
&+2^{p-1}\int_{\Omega}\int_{|x-y|>1}\frac{|f(y)|^p}{|x-y|^{d+sp}}\dd(y)^{-\be}\dd(x)^{-\al}\,dx\,dy\\
&\leq C\int_{\Omega}\frac{|f(x)|^p}{\dd(x)^{\al+\be}}\,dx,
\end{align*}
since, by the first lemma from \cite{MR312232} and the Muckenhoupt $A_1$ property we have
$$
\int_{|x-y|>1}\frac{\dd(x)^{-\al}}{|x-y|^{d+sp}}\,dx\leq C_1(d,s,p)M\dd^{-\al}(x)\leq C_2(d,\al,s,p)\dd(x)^{-\al}
$$
and the symmetric relation is valid for the second weight. Furthermore, analyzing the proof of the aforementioned inequality from \cite{MR312232} reveals that the constants are of order $O(1/(1-s))$, as $s\rightarrow 1^-$, which will contribute to the last statement of the Theorem.

Therefore, it suffices to focus on the integral $I_2$. Recall  the well-known pointwise estimate for Sobolev functions (see \cite{MR1226425}, \cite{MR312232}, \cite{MR1470421}), that is
$$
|u(x)-u(y)|\leq C|x-y|\left(M|\nabla u|(x)+M|\nabla u|(y)\right),\,\text{ for almost all }x,y\in\R^d,
$$
where $C=C(d)$.  This holds for $u\in W^{1,p}_{loc}(\R^d)$. Therefore, using the above, we can estimate $I_2$ as follows,
\begin{align*}
 I_2&\leq\int_{\R^d}\int_{|x-y|<1}\frac{|\Lambda f(x)-\Lambda f(y)|^p}{|x-y|^{d+sp}}\dd(x)^{-\al}\dd(y)^{-\be}\,dy\,dx\\ 
 &\leq C^p 2^{p-1}\int_{\R^d}\frac{\left(M|\nabla\Lambda f|(x)\right)^p}{\dd(x)^{\al}}\,dx\int_{|x-y|<1}\frac{\dd(y)^{-\be}\,dy}{|x-y|^{d-(1-s)p}}\\
 &+C^p 2^{p-1}\int_{\R^d}\frac{\left(M|\nabla\Lambda f|(y)\right)^p}{\dd(y)^{\be}}\,dy\int_{|x-y|<1}\frac{\dd(x)^{-\al}\,dx}{|x-y|^{d-(1-s)p}}.
\end{align*}
Again by \cite{MR312232} and the Muckenhoupt $A_1$ property, we have
\begin{align*}
 \int_{|x-y|<1}\frac{\dd(y)^{-\be}\,dy}{|x-y|^{d-(1-s)p}}&\leq C_3(d,s,p)M\dd^{-\be}(x)\leq C_4(d,s,p,\be)\dd^{-\be}(x)   
\end{align*}
and similarly $\int_{|x-y|<1}\frac{\dd(x)^{-\al}\,dx}{|x-y|^{d-(1-s)p}}\leq C_5(d,s,p,\al)\dd(y)^{-\al}$. In consequence, $I_2$ does not exceed a constant times
\begin{align*}
 \int_{\R^d}\frac{\left(M|\nabla\Lambda f|(x)\right)^p}{\dd(x)^{\al+\be}}\,dx\leq C_6\int_{\R^d}\frac{|\nabla\Lambda f(x)|^p}{\dd(x)^{\al+\be}}\,dx\leq C_7\|f\|^p_{W^{1,p}_{\al+\be}(\Omega)},   
\end{align*}
with $C_6$ and $C_7$ depending on $d,s,p,\al,\be$ and $\Omega$. In the first passage above, we used the well-known fact that the Hardy--Littlewood maximal operator is bounded on $L^p(\R^d,w)$, when $p>1$ and $w\in A_p$.

Now, assuming \eqref{condition2} instead of \eqref{condition1}, together with $|\al|+|\be|>0$, we first use a simple observation that, when $\al\be\geq 0$, one has $[f]_{W^{s,p}_{\al,\be}(\Omega)}\leq[f]_{W^{s,p}_{\al+\be,0}(\Omega)}$. Indeed, this follows from the application of Hölder's inequality with exponents $(\al+\be)/\al$ and $(\al+\be)/\be$. Then, we proceed as in the first part of the proof to bound $[f]_{W^{s,p}_{\al+\be,0}(\Omega)}$ by $[f]_{W^{1,p}_{\al+\be}(\Omega)}$. 

To end the proof, it suffices to notice that the property $$\limsup_{s\rightarrow 1^-}(1-s)[f]^p_{W^{s,p}_{\al,\be}(\Omega)}<\infty$$ clearly follows from obtained estimates.
\end{proof}

Before proving our first main result, we need a simple technical lemma concerning the continuity of weighted Gagliardo seminorms. For  reader's  convenience, we provide an easy proof.
\begin{lem}\label{lem1}
Let $f\colon\Omega\to\R^d$ be measurable and suppose that the weighted Gagliardo seminorms \eqref{weightedgagliardoalfabeta} are finite for $-\infty\leq\al_1<\al<\al_2\leq\infty$, $-\infty\leq\be_1<\be<\be_2\leq\infty$ and $0\leq s_1<s<s_2\leq 1$. Then, the function
$$
(s,\al,\be)\mapsto [f]^p_{W^{s,p}_{\al,\be}(\Omega)}
$$
is continuous on $(s_1,s_2)\times (\al_1,\al_2)\times(\be_1,\be_2)$. 
\end{lem}
\begin{proof}
  Let $\varepsilon>0$ be sufficiently small. Assume that $\al,\be\geq 0$, if one of the exponents is negative, the proof is similar. We have
  \begin{align*}
   [f]^p_{W^{s,p}_{\al,\be}(\Omega)}&=\left(\int_{\{d_{\Omega}(x)<1\}}+\int_{\{d_{\Omega}(x)>1\}}\right)\left(\int_{\{d_{\Omega}(y)<1\}}+\int_{\{d_{\Omega}(y)>1\}}\right)\frac{|f(x)-f(y)|^p}{|x-y|^{d+sp}}\\
   &\times d_{\Omega}(x)^{-\al}d_{\Omega}(y)^{-\be}\,dy\,dx\\
   &\leq [f]^p_{W^{s,p}_{\al+\varepsilon,\be+\varepsilon}(\Omega)}+[f]^p_{W^{s,p}_{\al-\varepsilon,\be-\varepsilon}(\Omega)} + [f]^p_{W^{s,p}_{\al+\varepsilon,\be-\varepsilon}(\Omega)}+[f]^p_{W^{s,p}_{\al-\varepsilon,\be+\varepsilon}(\Omega)}.
  \end{align*}
  Dividing the range of integration into cases $|x-y|>1$ and $|x-y|<1$, we obtain that the above does not exceed 
  \begin{align*}
  &[f]^p_{W^{s+\varepsilon,p}_{\al+\varepsilon,\be+\varepsilon}(\Omega)}+[f]^p_{W^{s+\varepsilon,p}_{\al-\varepsilon,\be-\varepsilon}(\Omega)}+ [f]^p_{W^{s-\varepsilon,p}_{\al+\varepsilon,\be+\varepsilon}(\Omega)}+[f]^p_{W^{s-\varepsilon,p}_{\al-\varepsilon,\be-\varepsilon}(\Omega)}\\
  &+[f]^p_{W^{s+\varepsilon,p}_{\al+\varepsilon,\be-\varepsilon}(\Omega)}+[f]^p_{W^{s-\varepsilon,p}_{\al+\varepsilon,\be-\varepsilon}(\Omega)}+ [f]^p_{W^{s+\varepsilon,p}_{\al-\varepsilon,\be+\varepsilon}(\Omega)}+[f]^p_{W^{s-\varepsilon,p}_{\al-\varepsilon,\be+\varepsilon}(\Omega)}<\infty.
  \end{align*}
  Thus, for $(s',\al',\be')\rightarrow (s,\al,\be)$, we may  use the above estimation to justify the Lebesgue Dominated Convergence theorem and obtain that $[f]^p_{W^{s',p}_{\al',\be'}(\Omega)}\longrightarrow [f]^p_{W^{s,p}_{\al,\be}(\Omega)}$. 
\end{proof}
\begin{proof}[Proof of Theorem \ref{bbmweighted}]
 Let $f\in C_c^2(\overline{\Omega})$ and let $0<\theta<1$. We have
$$I_1\leq [f]^p_{W^{s,p}_{\al,\be}(\Omega)}=I_1+I_2,$$
where
$$
I_1=\int_{\Omega}\int_{B(x,\theta \dd(x))}\frac{|f(x)-f(y)|^p}{|x-y|^{d+sp}}d_{\Omega}(x)^{-\al}d_{\Omega}(y)^{-\beta}\,dy\,dx
$$
and
$$
I_2=\int_{\Omega}\int_{\Omega\setminus B(x,\theta d(x))}\frac{|f(x)-f(y)|^p}{|x-y|^{d+sp}}d_{\Omega}(x)^{-\al}d_{\Omega}(y)^{-\beta}\,dy\,dx.
$$
If $|x-y|<\theta \dd(x)$, then $(1-\theta)\dd(x)<\dd(y)<(1+\theta)\dd(x)$, therefore
\begin{align*}
 I_1\geq (1+\theta\,\sgn\be)^{-\be}\int_{\Omega}\frac{\,dx}{\dd(x)^{\al+\be}}\int_{B(x,\theta \dd(x))}\frac{|f(x)-f(y)|^p}{|x-y|^{d+sp}}\,dy.   
\end{align*}
It was proved in  \cite[(6)]{BBM01} (see also \cite[Proof of Theorem 1.1]{MR4356180})  that, for $f\in C^2(\Omega)$,
\begin{equation}\label{bbmlimit}
\lim_{s\rightarrow 1^-}(1-s)\int_{B(x,\theta \dd(x))}\frac{|f(x)-f(y)|^p}{|x-y|^{d+sp}}\,dy=K_{d,p}|\nabla f(x)|^p,
\end{equation}
with $K_{d,p}$ defined by \eqref{Kdp}. Moreover, the convergence is dominated. When $\Omega$ is bounded, we can simply use the Lipschitz property of $f$ to obtain that
\begin{align}\label{bound}
(1-s)&\int_{B(x,\theta \dd(x))}\frac{|f(x)-f(y)|^p}{|x-y|^{d+sp}}\,dy\leq C\dd(x)^{(1-s)p}\\
\nonumber&\leq C\left(\diam\Omega\right)^{(1-s)p}\leq C\max\{1,\diam\Omega\}^p. 
\end{align}
Since $\al+\be<\ucodim_A(\partial\Omega)$, the weight $\dd^{-\al-\be}$ is integrable over $\Omega$ and  from the Lebesgue Dominated Convergence Theorem it follows that
\begin{equation}\label{I}
\lim_{s\rightarrow 1^-}(1-s)I=K_{d,p}\int_{\Omega}\frac{|\nabla f(x)|^p}{\dd(x)^{\al+\be}}\,dx,
\end{equation}
where
$$
I:=\int_{\Omega}\frac{dx}{\dd(x)^{\al+\be}}\int_{B(x,\theta \dd(x))}\frac{|f(x)-f(y)|^p}{|x-y|^{d+sp}}\,dy.
$$

However, when $\Omega$ is unbounded, the argument is slightly more involved. Let $K=\supp f\subset\overline{\Omega}$. We have 
\begin{align*}
 I=\int_{K}\frac{dx}{\dd(x)^{\al+\be}}\int_{B(x,\theta \dd(x))}\frac{|f(x)-f(y)|^p}{|x-y|^{d+sp}}\,dy+\int_{\Omega\setminus K}\frac{dx}{\dd(x)^{\al+\be}}\int_{K\cap B(x,\theta \dd(x))}\frac{|f(y)|^p}{|x-y|^{d+sp}}\,dy.   
\end{align*} 
Since $d_{\Omega}^{-\al-\be}\in L^1_{loc}(\overline{\Omega}$), for the first integral above we use a bound similar to \eqref{bound}. To deal with the second, we define $\Omega_1=\{x\in\Omega\setminus K\colon \dist(x,K)<1\}$ and $\Omega_2=\Omega\setminus\Omega_1$. The integral over $\Omega_1$ can again be treated like in \eqref{bound}, and the integral over $\Omega_2$ tends to zero, when multiplied by $1-s$, because for $x\in\Omega_2$ and $y\in K$ we have $|x-y|\geq 1$, hence $|x-y|^{-d-sp}\leq |x-y|^{-d-s_0 p}$ for $s\in (s_0,1)$, which yields the bound independent on $s$.

Hence, summarizing the above and using \eqref{bbmlimit} we obtain that 
$$
\liminf_{s\rightarrow 1^-}(1-s)I_1\geq (1+\theta\,\sgn\be)^{-\be}K_{d,p}\int_{\Omega}\frac{|\nabla f(x)|^p}{d_{\Omega}(x)^{\al+\be}}\,dx.
$$
Conversely, we analogously have 
$$\limsup_{s\rightarrow 1^-}(1-s)I_1\leq (1-\theta\,\sgn\be)^{-\be}K_{d,p}\int_{\Omega}\frac{|\nabla f(x)|^p}{d_{\Omega}(x)^{\al+\be}}\,dx.$$
Next, we will show that $\displaystyle\lim_{s\rightarrow 1^-}(1-s)I_2=0$. To this end, we write $s=\lambda s+(1-\lambda)s$ for sufficiently small $\lambda\in(0,1)$ and observe that
\begin{align*}
 I_2&=\int_{\Omega}\int_{\Omega\setminus B(x,\theta d(x))}\frac{|f(x)-f(y)|^p}{|x-y|^{d+s(1-\lambda)p}|x-y|^{s\lambda p}}d_{\Omega}(x)^{-\al}d_{\Omega}(y)^{-\beta}\,dy\,dx\\
 &\leq\theta^{-s\lambda p}\int_{\Omega}\int_{\Omega}\frac{|f(x)-f(y)|^p}{|x-y|^{d+s(1-\lambda)p}}d_{\Omega}(x)^{-\al-s\lambda p}d_{\Omega}(y)^{-\beta}\,dy\,dx.
\end{align*}
If $\lambda$ is appropriately small, we have $s(1-\lambda)>s_0$ and, by openness of $U$, $\al+s\lambda p\in U$ and the above is bounded for $s\rightarrow 1^-$ by Lemma \ref{lem1}, as a $W^{1-\lambda,p}_{\al+\lambda p,\be}(\Omega)$ - seminorm of $f$. In consequence, $(1-s)I_2\rightarrow 0$, when $s\rightarrow 1^-$. 

Overall, we get for any $0<\theta<1$ that
\begin{align*}
(1+\theta\,\sgn\be)^{-\be}K_{d,p}\int_{\Omega}\frac{|\nabla f(x)|^p}{d_{\Omega}(x)^{\al+\be}}\,dx&\leq \liminf_{s\rightarrow 1^-}(1-s)[f]^p_{W^{s,p}_{\al,\be}}\\
&\leq\limsup_{s\rightarrow 1^-}(1-s)[f]^p_{W^{s,p}_{\al,\be}(\Omega)}\\
&\leq (1-\theta\,\sgn\be)^{-\be}K_{d,p}\int_{\Omega}\frac{|\nabla f(x)|^p}{d_{\Omega}(x)^{\al+\be}}\,dx
\end{align*}
and the statement of the theorem for $f\in C_c^2(\overline{\Omega})$ follows by letting $\theta\rightarrow 0^+$. For general $f\in W^{1,p}_{\al+\be}(\Omega)$ and $\Omega$ being an extension domain, \eqref{limits1} is a consequence of density of $C_c^2(\overline{\Omega})$ in $W^{1,p}_{\al+\be}(\Omega)$.
Indeed, by \cite[Theorem 1.1]{MR2072105}, $C_c^{2}(\R^d)$ is dense in $W^{1,p}(\R^d,\dd^{-\al-\be})$ and the existence of the extension operator guarantees the desired result, by taking the restriction of the approximating functions to $\overline{\Omega}$.  

In order to end the proof, it suffices to show that for $p>1$ and $f\in L^{p}(\Omega,d_{\Omega}^{-\al-\be})$, if
$$
\liminf_{s\rightarrow 1^{-}}(1-s)\int_{\Omega}\int_{\Omega}\frac{|f(x)-f(y)|^p}{|x-y|^{d+sp}}d_{\Omega}(x)^{-\al}d_{\Omega}(y)^{-\be}\,dy\,dx<\infty,
$$
then $f\in W^{1,p}_{\al+\be}(\Omega)$.

Let $\mathcal{W}=\{Q_j\}_{j\in\mathbb{N}}$ be a Whitney decomposition of $\Omega$ into open cubes. In particular, when $x\in Q_n$, then $l(Q_n)\approx\dd(x)$, where $l(Q_n)$ denotes the length of the side of the cube $Q_n$ (by the symbol $\approx$ we mean here and elsewhere the comparability with constants independent of functions and varying parameters). Thus, we have
\begin{align*}
 [f]^p_{W^{s,p}_{\al,\be}(\Omega)}&\geq\sum_j\int_{Q_j}\int_{Q_j}\frac{|f(x)-f(y)|^p}{|x-y|^{d+sp}}d_{\Omega}(x)^{-\al}d_{\Omega}(y)^{-\be}\,dy\,dx\\
 &\approx \sum_{j}l(Q_j)^{-\al-\be}\int_{Q_j}\int_{Q_j}\frac{|f(x)-f(y)|^p}{|x-y|^{d+sp}}\,dy\,dx.
\end{align*}
By Fatou's lemma, for some constant $C=C(\al,\be,d)>0$,
\begin{align*}
 \liminf_{s\rightarrow 1^{-}}(1-s)&[f]^p_{W^{s,p}_{\al,\be}(\Omega)}\geq C\sum_j l(Q_j)^{-\al-\be}\liminf_{s\rightarrow 1^-}(1-s)[f]^p_{W^{s,p}(Q_j)}.
\end{align*}
Therefore, $\liminf_{s\rightarrow 1^-}(1-s)[f]^p_{W^{s,p}(Q_j)}$ is finite for every $Q\in\mathcal{W}$, hence, since $f\in L^p(Q_j)$, by the Bourgain--Brezis--Mironescu characterization of Sobolev spaces from \cite{BBM01} we have $f\in W^{1,p}(Q_j)$ and
$$
\lim_{s\rightarrow 1^-}(1-s)[f]^p_{W^{s,p}(Q_j)}=K_{d,p}\int_{Q_j}|\nabla f(x)|^p\,dx.
$$
In consequence,
\begin{align*}
 \liminf_{s\rightarrow 1^{-}}(1-s)[f]^p_{W^{s,p}_{\al,\be}(\Omega)}&\geq C K_{d,p}\sum_{j}l(Q_j)^{-\al-\be}\int_{Q_j}|\nabla f(x)|^p\,dx\\
 &\approx \sum_{j}\int_{Q_j}\frac{|\nabla f(x)|^p}{d_{\Omega}(x)^{\al+\be}}\,dx=\int_{\Omega}\frac{|\nabla f(x)|^p}{d_{\Omega}(x)^{\al+\be}}\,dx
\end{align*}
and we obtain that the latter is finite and $f\in W^{1,p}_{\al+\be}(\Omega)$. The proof is complete.
\end{proof}

\begin{proof}[Proof of Theorem \ref{thm2}]
First, we will prove the lower bound for the limit defined in the statement of the Theorem. By the weighted fractional Hardy inequality from \cite[Corollary 7]{sharpweighted}, we have
$$
 \int_{\R^d}\int_{\R^d}\frac{|f(x)-f(y)|^p}{|x-y|^d|x|^{\al}|y|^{\al}}\,dy\,dx\geq\mathcal{C}(d,p,\al)\int_{\R^d}\frac{|f(x)|^p}{|x|^{2\al}}\,dx,   
$$
where the constant $\mathcal{C}(d,p,\al)$ is given by
$$
\mathcal{C}(d,p,\al)=2\meas\int_{0}^1\frac{r^{\al-1}\left(1-r^{\frac{d-2\al}{p}}\right)^p}{1-r^2}\,dr.
$$
Integrating by parts and doing basic calculations, it is easy to show that $$\displaystyle\lim_{\al\rightarrow 0^+}\al\,\mathcal{C}(d,p,\al)=2\meas,$$
hence, we automatically obtain that
$$
\liminf_{\alpha\rightarrow 0^+}\al\int_{\R^d}\int_{\R^d}\frac{|f(x)-f(y)|^p}{|x-y|^d|x|^{\al}|y|^{\al}}\,dy\,dx\geq 2\meas\int_{\R^d}|f(x)|^p\,dx.
$$
To prove the upper bound, we split the integration range as in \cite[Proof of Theorem 3]{MR1940355} as follows,
\begin{align*}
 \int_{\R^d}\int_{\R^d}\frac{|f(x)-f(y)|^p}{|x-y|^d|x|^{\al}|y|^{\al}}\,dy\,dx&\leq 2\Bigg[\left(\int_{\R^d}\int_{|y|\geq 2|x|}\frac{|f(x)|^p}{|x-y|^d|x|^{\al}|y|^{\al}}\,dy\,dx\right)^{\frac{1}{p}}\\
 &+\left(\int_{\R^d}\int_{|y|\geq 2|x|}\frac{|f(y)|^p}{|x-y|^d|x|^{\al}|y|^{\al}}\,dx\,dy\right)^{\frac{1}{p}}\Bigg]^p\\
 &+2\int_{\R^d}\int_{|x|<|y|<2|x|}\frac{|f(x)-f(y)|^p}{|x-y|^d|x|^{\al}|y|^{\al}}\,dy\,dx\\
 &=:2\left(I_1^{\frac{1}{p}}+I_2^{\frac{1}{p}}\right)^p+2I_3.
\end{align*}
First, we will deal with the integral $I_1$. We have
\begin{align*}
 I_1&\leq\int_{\R^d}\frac{|f(x)|^p}{|x|^{\al}}\,dx\int_{|y-x|\geq |x|}\frac{\,dy}{|x-y|^d|y|^{\al}}\\
 &=\int_{\R^d}\frac{|f(x)|^p}{|x|^{2\al}}\,dx\int_{|y-e_1|\geq 1}\frac{\,dy}{|y-e_1|^d|y|^{\al}},
\end{align*}
by a standard argument of rotation and translation in the inner integral (here $e_1=(1,0,\dots,0)$ is the unit vector). Let $R>1$. Then
\begin{align*}
 \int_{|y-e_1|\geq 1}\frac{\,dy}{|y-e_1|^d|y|^{\al}}&=\int_{|y-e_1|\geq 1,\,|y|\leq R}\frac{\,dy}{|y-e_1|^d|y|^{\al}}+\int_{|y-e_1|\geq 1,\,|y|>R}\frac{\,dy}{|y-e_1|^d|y|^{\al}}.   
\end{align*}
The first term above is clearly bounded for $\al\rightarrow 0^+$, so it suffices to estimate the second one. Notice that by triangle inequality, polar coordinates and integration by parts,
\begin{align*}
 \al\int_{|y-e_1|\geq 1,\,|y|>R}\frac{\,dy}{|y-e_1|^d|y|^{\al}}&\leq\al\int_{|y|>R}\frac{\,dy}{(|y|-1)^d|y|^{\al}}=\al\meas\int_{R}^{\infty}\frac{r^{d-\al-1}}{(r-1)^d}\,dr\\
 &=\meas\left[R^{-\al}\left(\frac{R}{R-1}\right)^d+\int_{R}^{\infty}r^{-\al}\frac{d}{dr}\left(\frac{r}{r-1}\right)^d\,dr\right]\\
 &\longrightarrow \meas,
\end{align*}
when $\al\rightarrow 0^+$.

Since $|y|\geq 2|x|$ implies $|y-x|\geq\frac{|y|}{2}$, for the integral $I_2$ we have
\begin{align*}
 I_2\leq2^d\int_{\R^d}\frac{|f(y)|^p}{|y|^{d+\al}}\,dy\int_{|x|\leq\frac{|y|}{2}}\frac{\,dx}{|x|^{\al}}=\frac{2^{2d-\al}\meas}{d-\al}\int_{\R^d}\frac{|f(y)|^p}{|y|^{2\al}}\,dy,   
\end{align*}
so $\displaystyle\lim_{\al\rightarrow 0^+}\al\,I_2=0$.
Finally, let us consider the integral $I_3$. Observe that when $|x|<|y|<2|x|$, by triangle inequality it holds 
$$|x-y|\leq |x|+|y|<\min\{3|x|,2|y|\}\leq\sqrt{6|x||y|},$$
therefore
$$
I_3\leq 6^{\al}\int_{\R^d}\int_{|x|<|y|<2|x|}\frac{|f(x)-f(y)|^p}{|x-y|^{d+2\al}}\,dy\,dx.
$$
The latter expression is a part of the $W^{2\al,p}(\R^d)$-seminorm and it was shown in \cite[Proof of Theorem 3]{MR1940355} that it is of order $o(1/\al)$. Notice that $C_c^1(\R^d)\subset W^{2\al,p}(\R^d)$. In consequence, we have $\displaystyle\lim_{\al\rightarrow 0^+}\al\,I_3=0$, hence,
$$
\limsup_{\al\rightarrow 0^+}\al\int_{\R^d}\int_{\R^d}\frac{|f(x)-f(y)|^p}{|x-y|^d|x|^{\al}|y|^{\al}}\,dy\,dx\leq 2\meas\int_{\R^d}|f(x)|^p\,dx.
$$
The above combined with the first part ends the proof.
\end{proof}
\begin{rem} In the above proof it suffices to assume that $f$ is bounded with compact support (to satisfy the Hardy inequality from \cite[Corollary 7]{sharpweighted}) and $f\in\bigcup_{0<s<1}W^{s,p}(\R^d)$. In particular, taking $f=\chi_E$, the indicator function of a bounded Borel set $E$, we obtain that
$$
\lim_{\al\rightarrow 0^+}\al\int_{E}\int_{E^c}\frac{\,dy\,dx}{|x-y|^{d}|x|^{\al}|y|^{\al}}=\meas|E|,
$$
provided that $\chi_E\in\bigcup_{0<s<1}W^{s,p}(\R^d)$. The latter condition is related to the dimension of the boundary of $E$, see \cite{MR3032092}.
\end{rem}

\begin{proof}[Proof of Theorem \ref{mazyashapolimitd}]
The substitution $x\mapsto T(x),\,y\mapsto T(y)$ (see Subsection \ref{Section 3.1}) and the result of Theorem \ref{thm2} yields
\begin{align*}
\lim_{\al\rightarrow d^-}\left(d-\al\right)&\int_{\R^d}\int_{\R^d}\frac{|f(x)-f(y)|^p}{|x-y|^d|x|^{\al}|y|^{\al}}\,dy\,dx\\
&=\lim_{\al\rightarrow d^-}\left(d-\al\right)\int_{\R^d}\int_{\R^d}\frac{|Tf(x)-Tf(y)|^p}{|x-y|^d|x|^{d-\al}|y|^{d-\al}}\,dy\,dx\\
&=2\meas\int_{\R^d}|Tf(x)|^p\,dx\\
&=2\meas\int_{\R^d}\frac{|f(x)|^p}{|x|^{2d}}\,dx.
\end{align*}
\end{proof}
\noindent

\textbf{Acknowledgement.} The authors would like to thank the anonymous referee for numerous comments, which led to an improvement of the manuscript.

\def\polhk#1{\setbox0=\hbox{#1}{\ooalign{\hidewidth
  \lower1.5ex\hbox{`}\hidewidth\crcr\unhbox0}}}
  \def\polhk#1{\setbox0=\hbox{#1}{\ooalign{\hidewidth
  \lower1.5ex\hbox{`}\hidewidth\crcr\unhbox0}}} \def\cprime{$'$}
  \def\cprime{$'$} \def\cprime{$'$} \def\cprime{$'$} \def\cprime{$'$}
  \def\cprime{$'$}

\end{document}